\documentclass[a4paper, 12pt]{article}
\usepackage[margin=1in]{geometry}

\usepackage[utf8]{inputenc}
\usepackage[english]{babel}

\usepackage{import}

\usepackage{tikz-cd}
\usetikzlibrary{arrows,decorations.pathreplacing,calligraphy,patterns,bending,hobby,shapes}
\tikzset{tail reversed/.code={\pgfsetarrowsstart{tikzcd to}}}

\usepackage{wrapfig}
\usepackage{subcaption}
\usepackage{xfrac}
\usepackage[makeroom]{cancel}
\usepackage{array}
\usepackage{tabularx}
\newcolumntype{C}{>{\centering\arraybackslash$} p{1em} <{$}}
\usepackage{booktabs}
\usepackage{float}
\usepackage{faktor}
\usepackage{enumitem}
\usepackage[toc,page]{appendix}
\usepackage{hyperref}

\usepackage[T1]{fontenc}
\usepackage{hyperref} 
\usepackage{amsfonts}
\usepackage{graphicx}
\usepackage{pdfpages}

\usepackage{colortbl}
\usepackage{quoting}

\usepackage{amssymb}
\usepackage{amsmath}
\usepackage{amsthm}

\usepackage{mathtools}

\usepackage[libertine]{newtxmath}
\usepackage{libertine}
\usepackage{inconsolata}

\newcommand{\Z}{\mathbb{Z}}

\newcommand{\CC}{\mathbb{C}}

\newcommand{\GL}{\textnormal{GL}}
\newcommand{\so}{\mathfrak{so}}
\newcommand{\SO}{\textnormal{SO}}

\newcommand{\Sp}{\textnormal{Sp}}

\newcommand{\circled}[1]{\raisebox{.5pt}{\textcircled{\raisebox{-.9pt} {#1}}}}
\newcommand{\A}{\mathcal{A}}

\makeatletter
\newcommand{\shortoverline}{\mathpalette\@thickbar}
\newcommand{\@thickbar}[2]{{#1\mkern0mu\vbox{
    \sbox\z@{$#1#2\mkern-1.5mu$}%
    \dimen@=\dimexpr\ht\tw@-\ht\z@+9\p@\relax 
    \hrule\@height0.6\p@ 
    \vskip\dimen@
    \box\z@}}
}
\makeatother

\renewcommand{\arraystretch}{1.5}

\makeatletter
\def\th@plain{%
  \thm@notefont{}
  \itshape 
}
\def\th@definition{%
  \thm@notefont{}
  \normalfont 
}
\makeatother

\theoremstyle{definition}
\newtheorem{de}[equation]{Definition}
\newtheorem{ejemplo}[equation]{Example}

\newtheorem{nota}[equation]{Note}

\theoremstyle{plain}

\newtheorem{lem}[equation]{Lemma}
\newtheorem{cor}[equation]{Corollary}
\newtheorem{prop}[equation]{Proposition}

\usepackage{calc}
\newlength{\mylength}
\newenvironment{ej}{
    \small
    \parshape 1 \parindent \mylength
    \begin{ejemplo}
    }{
    \end{ejemplo}
}

\newenvironment{note}{
    \small
    \parshape 1 \parindent \mylength
    \begin{nota}
    }{
    \end{nota}
}

\DeclareMathOperator{\Par}{Par}
\DeclareMathOperator{\Irr}{Irr}
\DeclareMathOperator{\rect}{rect}

\usepackage{ytableau}
    
    \newcommand{\yTableau}{\ytableausetup{boxsize=1.5em, aligntableaux=center}\scriptsize}
    \newcommand{\yTableauTwo}{\ytableausetup{boxsize=1.3em, aligntableaux=center}\scriptsize}

\title{Bender--Knuth involutions for types B and C}

\author{
\'{A}lvaro Guti\'{e}rrez\footnote{\href{a.gutierrezcaceres@bristol.ac.uk}{a.gutierrezcaceres@bristol.ac.uk} \\ Department of Mathematics, University of Bonn, Bonn \\ School of Mathematics, University of Bristol, Bristol \\ This project started as part of a MSc thesis project in the University of Bonn and was partially funded by the University of Bristol Research Training Support Grant.}}

\date{May 28, 2024}

\begin{document}
\numberwithin{equation}{section}
\maketitle

\begin{abstract}
    We show that the combinatorial definitions of King and Sundaram of the symmetric polynomials of types B and C are indeed symmetric, in the sense that they are invariant by the action of the Weyl groups. Our proof is combinatorial and inspired by Bender and Knuth's classic involutions for type A.
\end{abstract}
\tableofcontents

\section{Introduction}
Symmetric polynomials are a central object of study in two branches of mathematics. On the one side, for combinatorialists, they are generating functions. On the other side, for representation theorists, they are characters of representations. The interplay between these two points of view is best presented with a diagram (see Figure \ref{fig: survey}).

\begin{figure}[h]
\vspace{-1em}
\[\small
\begin{tikzcd}[row sep=0em]
	&&& {\substack{\text{Jacobi--Trudi-type}\\\text{formulas}}} \\
	{\substack{\text{Generating functions}\\\text{of tableaux}}} \\
	&&&&& {\substack{\text{Weyl's character}\\\text{formula}}} \\
	&& {\substack{\text{Weight functions}\\\text{of crystals}}}
	\arrow["{\substack{\text{Crystal theory}\\\,\\\,}}"{marking, allow upside down}, tail reversed, from=4-3, to=3-6]
	\arrow["{\substack{\,\\\,\\\,\\\text{Jeu de taquin}}}"{marking, allow upside down}, tail reversed, from=2-1, to=4-3]
	\arrow["{\substack{\text{Path enumeration}\\\,\\\,}}"{marking, allow upside down}, tail reversed, from=2-1, to=1-4]
	\arrow[tail reversed, dashed, from=1-4, to=3-6]
    \arrow[tail reversed, dashed, from=2-1, to=3-6]
\end{tikzcd}\]
\vspace{-2em}
    \caption{Different approaches to the topic in the literature.}
    \label{fig: survey}
\end{figure}

In type A, corresponding to the representation theory of $\GL_n$ and its combinatorics, the objects and relations in Figure \ref{fig: survey} are well understood.
\emph{Schur (symmetric) polynomials} arise as the generating functions of semistandard Young tableaux and as the (Weyl) characters of irreducible polynomial representations of $\GL_n$. They are also the generating functions of Gelfand--Tsetlin patterns \cite{StanleyEC2}. For other Lie groups such as $\SO_{2n+1}$ (type B) and $\Sp_{2n}$ (type C), we have different candidates for tableaux and patterns whose generating functions are the \emph{orthogonal (Schur) polynomials} and \emph{symplectic (Schur) polynomials}. In this work, we focus on the tableau defined by King for type C \cite{King} and Sundaram for type B \cite{SundaramOrthogonal}, who show that their generating functions verify some recursive algebraic formulas and deduce that they recover the irreducible characters of the corresponding Lie groups \cite{Sundaram, SundaramOrthogonal}.

Characters can be computed as the determinants of matrices whose entries are elementary symmetric polynomials. For type A this is the Jacobi--Trudi formula \cite{StanleyEC2}; for types B and C, see \cite[Lec.~24 and App.~A]{FH91}, \cite{KT87}. These determinants enumerate tableaux after an argument of Gessel and Viennot for type A \cite{GV}, \cite[Sec.~7.16]{StanleyEC2}; see \cite{FK97} for the type B and C analogues. 

On the other hand, crystal bases allow Kashiwara and Nakashima to propose their own tableaux definitions \cite{KN}. For type A, these recover the above combinatorics; for type C, they are seen to be in bijection with King tableaux in \cite{Sheats}, via an analogue of the jeu de taquin algorithm. We are not aware of analogous bijections for type B in the literature.\\

Without direct proofs, it is not immediately obvious that generating functions of tableaux are the correct candidates for describing characters of representations.
In particular, we ask whether these generating functions are invariant by the action of the Weyl group of corresponding type; $W(A_{n-1})\cong S_n$ and $W(B_n) \cong S_2\wr S_n \cong W(C_n)$. This is certainly true of Weyl characters, but we seek a direct and combinatorial proof.

For type A, a short argument by Bender and Knuth \cite{BenderKnuth}, \cite[Thm.~7.10.2]{StanleyEC2} shows that the number of semistandard Young tableaux of a fixed shape $\lambda$ with weight $x^\alpha$ coincides with the number of semistandard Young tableaux of shape $\lambda$ with weight $(j~j\!+\!1).x^\alpha$ for any simple transposition $(j~j\!+\!1)$, $j =1, ..., n-1$. This is done by constructing an involution $\smash{BK_j^\text{A}}$ which is known as a \emph{(type A) Bender--Knuth involution}. Therefore, the Schur polynomial of shape $\lambda$ is a \emph{symmetric} polynomial; an element of $\smash{\CC[x_1, ..., x_n]^{S_n}}$.
As a remark, these involutions do \emph{not} induce an action of $S_n$ on the set of tableaux of fixed shape, in general. Type A Bender--Knuth involutions are translated to Gelfand--Tsetlin patterns in \cite{BK95}; we review these constructions later in this text.

We introduce \emph{type C Bender--Knuth involutions} $BK^\text{C}_j$ and show combinatorially and directly for the first time that there is an involutory action of $S_2\wr S_n$ on the set of King patterns of top row $\lambda$. We conclude that symplectic polynomials lie in $\CC[x_1^\pm, ..., x_n^\pm]^{S_2\wr S_n}$. As a corollary, we get \emph{type B Bender--Knuth involutions} and the analogous result for orthogonal polynomials. The maps are later translated to tableaux. As expected, our involutions do not define an action of $S_2\wr S_n$ on the sets of patterns of a fixed shape (see \cite[Note 4.11]{GutiTFM}).

These results were first claimed by Sundaram in \cite{Sundaram}. However, as noted by Hopkins \cite{MathOverflow}, the original proof is incorrect and cannot be fixed in any simple way. (The claimed proof only shows that the functions are invariant by a set of elements of $S_2\wr S_n$ which do not form a generating set for the group.) A corollary of our result, also noted by Sundaram \cite{Sundaram, SundaramOrthogonal}, is that the symplectic and orthogonal polynomials define class functions on the set of diagonalizable elements in the algebraic groups of types B and C.\\

A first candidate for type C Bender--Knuth involutions is given as a composition of type A Bender--Knuth involutions. The resulting patterns may not be symplectic, so we post-compose with a \emph{rectification} map.
A `locality' argument allows us to reduce our proof to computing that $BK^\text{C}_2$ is an involution on a generic pattern when $n = 3$. This reduction is what enables us to conclude the result.

It is worth remarking that there are two ways of approaching this computation on generic patterns when $n=3$. One is to argue directly and `by hand', as we do. Alternatively, one could use a computer to check that $BK^\text{C}_2$ (as a tropical rational map) is involutory. This appears to be beyond the reach of computer algebra systems at the time of writing. 
But we were able to check that $\text{Trop}^{-1} BK^\text{C}_2$ (as a rational map) is involutory. It remains to argue that the order of $BK^\text{C}_2$ and the order of $\text{Trop}^{-1} BK^\text{C}_2$ coincide. This step is in general non-trivial (see e.g. \cite{grinbergRoby14}).\\

We recall some preliminary definitions in Section \ref{sec: preliminary}. We define tableaux and patterns for types A, B, and C in Section \ref{sec: tableaux}, and we define symmetric polynomials as their generating functions. We recall type A Bender--Knuth involutions and introduce the type B and C analogues in Section \ref{sec: bk}. Proving that these are involutions reduces to a computation, that we leave for Section \ref{sec: appendix bender-knuth}.

\section{Preliminary definitions}
\label{sec: preliminary}

Fix a natural number $n\ge 1$ throughout this work. We work over $\CC$. We will follow \cite[Ch. 7]{StanleyEC2} for the standard concepts on symmetric polynomials, and \cite{FH91} for Lie theory.

\subsection{Symmetric polynomials and partitions}
 
The \emph{space of symmetric polynomials} in $n$ variables is $\Lambda_n = \CC[x_1, ..., x_n]^{S_n}$, where the symmetric group $S_n$ acts by permuting the variables.
A \emph{partition} $\lambda$ of length smaller or equal to $n$ is an $n$-tuple of weakly decreasing non-negative integers. Let $\Par_n$ be the set of partitions of length smaller or equal to $n$.
Bases of $\Lambda_n$ are indexed by $\Par_n$.
We represent partitions through their \emph{Young diagram}, which we draw following the English convention.

\subsection{Lie groups and Weyl groups}

The set $\Irr(\GL_{n})$ of irreducible polynomial representations of $\GL_{n}$ is indexed by $\Par_{n}$.
The relationship between $\Irr(\GL_n)$ and $\Lambda_{n}$ is explained through the following result: the (Weyl) characters of irreducible polynomial representations of $\GL_{n}$ form a basis of $\Lambda_{n}$. The character of the irreducible representation indexed by $\lambda$ is the Schur polynomial $s_\lambda(x_1, ..., x_n)$, as defined combinatorially in Section \ref{sec: sym pols}.

The sets $\Irr(\Sp_{2n})$ and $\Irr(\SO_{2n+1})$ are also indexed by $\Par_n$. (The set of irreducible representations of the Lie algebra $\so(2n+1)$ is richer, and indexed by the set of partitions and \emph{half-partitions}. The representations indexed by half-partitions are called spin representations, and will not be modelled by the combinatorics of this document.)
The irreducible characters for $\SO_{2n+1}$ and $\Sp_{2n}$ are known as the \emph{symmetric polynomials of types B and C}. These will be defined purely combinatorially in Section \ref{sec: sym pols}, and referred to as \emph{orthogonal polynomials} and \emph{symplectic polynomials}, respectively. Note that these are not in $\Lambda_n$. Rather, they lie in the ring $\CC[x_1^\pm, ..., x_n^\pm]^W$ of Laurent polynomials invariant under the Weyl group of corresponding type (as a permutation group of the variables). In type A, these Laurent polynomials are polynomials, and the Weyl group of $\GL_{n}$ is $W(A_{n-1}) \cong S_{n}$; this is consistent with the above.
The Weyl groups of type B and C coincide, and are isomorphic to the wreath product $S_2\wr S_n$. We avoid wreath products in this work, and instead provide the following description of these Weyl groups as subgroups of $S_{2n}$: if we interpret $S_{2n}$ as the permutation group of the set $\{1, \bar{1}, 2, \bar{2}, ..., n, \bar{n}\}$, then $W(B_n)$ and $W(C_n)$ are isomorphic to the subgroup of $S_{2n}$ generated by $(1\ \bar{1})$ and the permutations $(j~~j\!+\!1)(\shortoverline{j}~~\shortoverline{j\!+\!1})$ for $j\in [n-1]$.

\section{Combinatorics and symmetric functions}
\label{sec: tableaux}

\subsection{Tableaux for types A, B, and C}

\newcommand{\X}{\mathcal{X}}
\newcommand{\Y}{\mathcal{Y}}
Fix $\lambda\in\Par_n$. Let $[\lambda] := \{(i,j)\in[n]\times\Z_{>0} : j \le \lambda_i\}$ be the set of its cells.
Let $\X$ be a totally ordered set. A tableau of shape $\lambda$ in the alphabet $\X$ is a function $T:[\lambda]\to \X$.
We say a tableau is \emph{semistandard} if $T(i,j)<T(i+1,j)$ and $T(i,j)\le T(i,j+1)$ whenever this makes sense.

The (set-wise) co-restriction of a map $f : A\to B$ to a subset $C\subset B$ is defined to be the restriction of $f$ to $f^{-1}(C)$.
    \newcommand{\SSYT}{\textnormal{SSYT}}
    \newcommand{\SOT}{\textnormal{SOT}}
    \newcommand{\KSpT}{\textnormal{KSpT}}
\begin{de}\label{def: symplectic tableaux}
Let $\A := \{1<\bar{1}<2<\bar{2}<\cdots<n<\bar{n}\}$, and $\A_\infty := \{1<\bar{1}<2<\bar{2}<\cdots<n<\bar{n}<\infty\}$ be two ordered sets.

\begin{itemize}
    \item[(A)] A \emph{semistandard Young tableaux} (on $n$ letters) of shape $\lambda$ is a semistandard tableau of shape $\lambda$ in the alphabet $[n] = \{1 < 2 < \cdots < n\}$.
    \item[(B)] A \emph{(Sundaram) orthogonal tableau} 
    $T$ (on $n$ letters) of shape $\lambda$ is a tableau of shape $\lambda$ in the alphabet $\A_\infty$ such that
    \begin{itemize}
        \item the co-restriction of $T$ to $\A$ defines a symplectic tableau (see below), and
        \item at most one cell per row (the right-most cell) takes the value $\infty$.
    \end{itemize}
    \item[(C)] A \emph{(King) symplectic tableau} $T$ (on $n$ letters) of shape $\lambda$ is a semistandard tableau of shape $\lambda$ in the alphabet $\A$ such that $T(i,j)\ge i$ for all $(i, j) \in [\lambda]$.
\end{itemize}
We let $\SSYT_n(\lambda)$, $\SOT_n(\lambda)$, and $\KSpT_n(\lambda)$ denote the sets of such tableaux.

The weight of a tableau $T : [\lambda] \to \X$ is the monomial
    \(
    x^T = \prod_{a\in \X} x_{T^{-1}(a)}
    \) 
    in the ring
    \(
    \CC[x_i : i\in\X]
    \).
We take the conventions $\smash{x^{\ }_{\bar{\imath}} = x_i^{-1}}$ and $x_\infty = 1$. That is, weights of 
semistandard Young tableaux are monomials in $\CC[x_1, ..., x_n]$, whereas weights of orthogonal and symplectic tableaux lie in
\(
\CC[x_1^{\pm}, ..., x_n^{\pm}] .
\)
\end{de}

\begin{ej} We present a semistandard Young tableau, an orthogonal tableau, and a symplectic tableau of shape $(3, 3, 2)$ and their weights.
\[
\arraycolsep2em
\begin{array}{ccc}
{\yTableau
\ytableaushort{113,234,34}}\quad x_1^2x_2x_3^3x_4^2 &
{\yTableau\ytableaushort{12{\infty},33{\infty},{\bar{3}}{\bar{3}}}}\quad x_1x_2 &
{\yTableau\ytableaushort{12{\bar{2}},33{\bar{3}},{\bar{3}}{\bar{3}}}}\quad x_1x_3^{-1}
\end{array}
\]
\end{ej}

\begin{note}\label{note: type B, C as type A}
    Co-restriction to $\A$ defines a bijection $\SOT_n(\lambda)\to\bigcup_\mu \KSpT_n(\mu)$, where $\mu$ ranges over the partitions which may be formed from $[\lambda]$ by removing at most one cell per row. 
    The order-preserving map $\A\to[2n]$ gives an inclusion $\KSpT_n(\lambda)\subseteq\SSYT_{2n}(\lambda)$.
\end{note}

\subsection{Patterns for types A, B, and C}\label{sec: patterns}

\newcommand{\GT}{\textnormal{GT}}

\begin{de}
    A \emph{Gelfand--Tsetlin pattern} (or GT pattern) with $n$ rows is a triangular array of non-negative integers $P = (P^{(n)}, ..., P^{(1)})$, with $P^{(k)} = (p_{1k}, ..., p_{kk})$ for $k \in[n]$, subject to the local inequalities of Figure \ref{fig:GTpattern}.
    We say $P^{(k)}$ for $k\in[n]$ are the rows of $P$ and we call $P^{(n)}$ the \emph{top row}. Note that $P^{(k)} \in \Par_k$ for each $k$. 
    Let $\GT_n(\lambda)$ be the set of GT patterns with $n$ rows and top row $\lambda$.
\end{de}
\begin{figure}[h]
    \centering\vspace{-2em}
    \[
    \arraycolsep=0.2em\def\arraystretch{1}
    \begin{array}{CCCCCCC}
    p_{14} & & p_{24} & & p_{34} & & p_{44} \\
    & p_{13} & & p_{23} & & p_{33} & \\
    & & p_{12} & & p_{22} & & \\
    & & & p_{11} & & &
    \end{array}
    \quad
    \quad
    \begin{tikzpicture}[baseline=0, y=1.7em]
        \node (ij) at (0,0) {$p_{i,j}$};
        \node (ij+1) at (-1,1) {$p_{i,j+1}$};
        \node (i+1j+1) at (1,1) {$p_{i+1,j+1}$};
        \node (i-1j-1) at (-1,-1) {$p_{i-1,j-1}$};
        \node (ij-1) at (1,-1) {$p_{i,j-1}$};
        
        \draw[white] (ij+1) -- (ij) node[midway, black, rotate=-45]{$\ge$};
        \draw[white] (i+1j+1) -- (ij) node[midway, black, rotate=45]{$\ge$};
        \draw[white] (i-1j-1) -- (ij) node[midway, black, rotate=45]{$\ge$};
        \draw[white] (ij-1) -- (ij) node[midway, black, rotate=-45]{$\ge$};
    \end{tikzpicture} \vspace{-1em}
    \]
    \caption{Left: the arrangement of a GT pattern of size 4. Right: the local inequalities.}
    \label{fig:GTpattern}
\end{figure}

We have a bijection $\SSYT_n(\lambda)\to\GT_n(\lambda)$, by letting $P^{(k)}$ be the shape of $T^{-1}[k]$ (see \cite[Sec.~7.10]{StanleyEC2}). In other words, $p_{i,j}$ counts the number of entries smaller or equal to $j$ in the $i$th row of $T$. See Example \ref{eg: patterns}.

Trough this bijection, the $j$th row sum $S_j := \sum_i p_{ij}$ of a pattern $P$ counts the number of entries smaller or equal to $j$ in the corresponding tableau. Therefore,
if the weight of a pattern $P$ is defined as the monomial $x^P := x_1^{S_1}x_2^{S_1-S_2}\cdots x_n^{S_n-S_{n-1}}$, the bijection is weight-preserving.

\newcommand{\KSpP}{\textnormal{KSpP}}
\newcommand{\SOP}{\textnormal{SOP}}
\begin{de}
A \emph{(King) symplectic pattern} is a Gelfand--Tsetlin pattern $P$ in which $p_{ij} = 0$ whenever $2i > j$ (see \cite{King}). We let $\KSpP_n(\lambda)$ be the set of symplectic patterns with $2n$ rows and top row $\lambda$. 
\end{de}
We think of these as ``half-triangular'' arrays by omitting the entries $p_{ij}$ with $2i>j$. See Example \ref{eg: patterns}.
We introduce the following definition.
\begin{de}
A  \emph{(Sundaram) orthogonal pattern} is a symplectic pattern in which top row entries might be circled. Let $P$ be an orthogonal pattern with $N$ rows. 
The shape $\lambda$ of $P$ is defined by $\lambda_i := p_{iN}+1$ if $p_{iN}$ is circled and $\lambda_i := p_{iN}$ otherwise.
For a partition $\lambda$, we let $\SOP_n(\lambda)$ be the set of orthogonal patterns with $2n$ rows and shape $\lambda$.
\end{de}

The maps from Note \ref{note: type B, C as type A} together with the bijection $\SSYT_n(\lambda)\leftrightarrow\GT_n(\lambda)$ give
maps $\KSpT_n(\lambda) \leftrightarrow \KSpP_n(\lambda) \subseteq \GT_{2n}(\lambda)$ and $\SOT_n(\lambda) \leftrightarrow \SOP_n(\lambda) \leftrightarrow \bigcup_\mu \KSpP_n(\mu)$.

Trough these bijections, given a pattern $P$ of type B or C, the difference $S_{2j-1} - S_{2j-2}$ counts the  number of entries equal to $j$ in the corresponding tableau, whereas $S_{2j} - S_{2j-1}$ counts the number of entries equal to $\bar{\jmath}$. Therefore, if the weight of a pattern $P$ is defined as the monomial $x^P := x_1^{2S_1-S_2}x_2^{2S_3-S_4-S_2}\cdots x_n^{2S_{2n-1}-S_{2n}-S_{2n-2}}$, the bijections are weight-preserving.

\begin{ej}
\label{eg: patterns}
We present patterns of top row (or shape) $\lambda = (3,2)$ for types A, B, and C.
\[
\begin{tabularx}{\textwidth}{>{\centering\arraybackslash\hsize=.27\hsize}X>{\centering\arraybackslash\hsize=.27\hsize}X>{\centering\arraybackslash\hsize=.27\hsize}X}
    {\yTableau\ytableaushort{113,23}}
    $\,\leftrightarrow\,$
    \arraycolsep=0.1em\def\arraystretch{1}
    \scriptsize
    $\begin{array}{CCCCC}
    3 & & 2 & & 0\\
    & 2 & & 1 &\\
    & & 2 & &
    \end{array}$ \quad&\quad
    {\yTableau\ytableaushort{1{\bar{1}}{\infty},{2}{\bar{2}}}}
    $\,\leftrightarrow\,$
    \arraycolsep=0.1em\def\arraystretch{1}
    \scriptsize
    $\begin{array}{CCCC|}
    \circled{2} &   & 2 &\\
      & 2 &   & 1 \\
      &   & 2 & \\
      &   &   & 1 
    \end{array}$ \quad&\quad
    {\yTableau\ytableaushort{1{\bar{1}}2,2{\bar{2}}}}
    $\,\leftrightarrow\,$
    \arraycolsep=0.1em\def\arraystretch{1}
    \scriptsize
    $\begin{array}{CCCC|}
    3 &   & 2 &\\
      & 3 &   & 1 \\
      &   & 2 & \\
      &   &   & 1 
    \end{array}$ 
    \end{tabularx}
\]
\end{ej}

\subsection{Symmetric polynomials as generating functions}
\label{sec: sym pols}

\begin{de}
Let $\lambda\in\Par_n$. The \emph{Schur polynomial} $s_\lambda$, the \emph{orthogonal polynomial} $o_\lambda$, and the \emph{symplectic polynomial} $sp_\lambda$ on $n$ letters and of shape $\lambda$ are defined as the generating functions of semistandard Young tableaux, orthogonal tableaux, and symplectic tableaux on $n$ letters and of shape $\lambda$, respectively. Explicitly, 
\[
s_\lambda(x_1, ..., x_n) =\!\!\!\!\! \sum_{T\in\SSYT_n(\lambda)}\!\!\!\!\! x^T,
\quad
o_\lambda(x_1, ..., x_n) =\!\!\!\!\! \sum_{T\in\SOT_n(\lambda)}\!\!\!\!\! x^T,
\quad\text{and}\quad
sp_\lambda(x_1, ..., x_n) =\!\!\!\!\! \sum_{T\in\KSpT_n(\lambda)}\!\!\!\!\! x^T.
\]
Equivalently, they are the generating functions of $\GT_n(\lambda)$, $\SOP_n(\lambda)$, and $\KSpP_n(\lambda)$.
\end{de}

\section{Bender--Knuth involutions}
\label{sec: bk}

We study type B and C analogues of the following elegant proof of Bender and Knuth.

\begin{prop}\label{lem: bender-knuth A}
    Schur polynomials on $n$ letters are $W(A_{n-1})$-symmetric.
\end{prop}
\begin{proof}[Sketch of proof.] We follow \cite{BenderKnuth}, \cite[Thm.~7.10.2]{StanleyEC2}.
    Let $(j~j\!+\!1)$ be a simple transposition.
    Given a tableau $T$ of shape $\lambda$ and weight $x^T$, we produce a tableau $BK_j^{\text{A}}(T)$ of shape $\lambda$ and weight 
    \[
    x^{BK_j^{\text{A}}(T)} = (j~j\!+\!1).x^T,
    \]
    where $S_n$ acts on $\CC[x_1, ..., x_n]$ by permuting the variables.
    To construct $\smash{BK^{\text{A}}_j}(T)$, begin by \emph{freezing} each pair of entries $\scriptsize\boxed{j}, \boxed{j+1}$ that lie in the same column of $T$ (these are usually referred to as $\{j, j+1\}$-vertical dominoes). The entries $\smash{\scriptsize\boxed{j}, \boxed{j+1}}$ of $T$ that are not yet frozen are called \emph{mutable}. For each row, the mutable entries of $T$ form a word $j^a(j+1)^b$, which is changed to a word $j^b(j+1)^a$ in $\smash{BK^{\text{A}}_j}(T)$.
    The map $BK_j^{\text{A}} : \SSYT_n(\lambda)\to \SSYT_n(\lambda)$ is an involution.
\end{proof}

We refer to the map $\smash{BK^{\text{A}}_j}$ defined in this proof as the $j$th \emph{type A Bender--Knuth involution}. See Figure \ref{fig: example BK}.

    \begin{figure}[h]
        \centering
        $\yTableauTwo
        \ytableaushort{1111222333444,222334444,3334}*[*(yellow)]{9+4,4+3,3}
        \mapsto
        \ytableaushort{1111222333334,222333444,4444}*[*(green!30)]{9+4,4+3,3}$
        \caption{We illustrate $\smash{BK^{\text{A}}_3}$. Highlighted, mutable entries of the tableau.}
        \label{fig: example BK}
    \end{figure}

The translation of type A Bender--Knuth involutions to GT patterns was studied in \cite{BK95} and can be described as follows: it only affects the $j$th row $P^{(j)} = (p_{i,j})_{i\le j}$, and sends each entry $p_{i,j}$ to
\[
p_{i,j}' := \min\{p_{i,j+1},p_{i-1,j-1}\} + \max\{p_{i+1,j+1},p_{i,j-1}\} - p_{i,j},
\]
where $\min$ and $\max$ ignore non-existing entries. That is,
\[
BK_j^\text{A}\left(P^{(n)}, ..., P^{(j)}, ..., P^{(1)}\right)
=
\left(P^{(n)}, ..., \big(p'_{1,j}, ..., p'_{j,j}\big), ..., P^{(1)}\right).
\]

\begin{ej}
To motivate this translation of $\smash{BK^{\text{A}}_j}$ to GT patterns, we take a closer look at the example of Figure \ref{fig: example BK}. Let $T$ be the tableau on the left, and let $P = ((p_{i,n})_{i\le n}, ..., (p_{i,1})_{i\le1})$ be its corresponding GT pattern. We consider the mutable $3^1 4^2$ word in the second row of $T$. It spans three columns, starting at (but not including) column 4 and ending at column 7.
\[
\yTableauTwo\ytableausetup{aligntableaux = top}
    \ytableaushort{1111222333444,222334444,3334}*[*(yellow)]{0,4+3}
\]
We can compute its starting column (in this case 4) as the maximum of (a) the number of entries smaller than 3 in the second row, $p_{2,2} = 3$, and (b) the number of entries smaller or equal to 4 in the third row, $p_{3,4} = 4$. Similarly, its ending column, 7, is the minimum of (c) the number of entries smaller than 3 in the first row, $p_{1,2} = 7$, and (d) the number of entries smaller or equal to 4 in the second row, $p_{2,4} = 9$.

Correspondingly, $p_{2,3}=5$ is sent by $\smash{BK^{\text{A}}_3}$ to $\min\{7,9\}+\max\{3,4\} - 5 = 6$.
\[
    \arraycolsep=0.1em\def\arraystretch{1}
    \scriptsize
    \begin{array}{CCCCCCC}
    13&&9&&4&&0\\
    &10&&\colorbox{yellow}{5}&&3&\\
    &&7&&3&&\\
    &&&4&&&
    \end{array}
    \mapsto
    \begin{array}{CCCCCCC}
    13&&9&&4&&0\\
    &12&&\colorbox{green!30}{6}&&0&\\
    &&7&&3&&\\
    &&&4&&&
    \end{array}
\]
\end{ej}

\noindent
We now show the analogue result for type C.

\begin{prop}\label{thm: bender-knuth C}
Symplectic polynomials in $n$ letters are $W(C_n)$-symmetric.
\end{prop}

We begin by proposing a candidate for type C Bender--Knuth involutions. An involutory action of $(1~\bar{1})$ on $\KSpP_n(\lambda)\subseteq\GT_{2n}(\lambda)$ is given by $BK^\text{A}_1$.
For any other generator of $W(C_n)$, write this permutation as a product of simple transpositions
with respect to the ordered set $\A = \{1<\bar{1}<\cdots<n<\bar{n}\}$. We multiply permutations right-to-left, and thus we get
\[
(j~~j\!+\!1)(\overline{j}~~\overline{j\!+\!1}) = (\overline{j}~~ j\!+\!1)(j\!+\!1~~\overline{j\!+\!1})(j~~ \overline{j})(\overline{j}~~ j\!+\!1).
\]
For each of these, perform a type A involution. Starting with any given pattern $P_0$, we get
\begin{equation}\label{eq: composition BK}
\begin{tikzcd}[column sep = 5em]
P_0 \arrow[r,  "BK_{2j}^{\text{A}}", maps to] \arrow[r,  "(\overline{j}\ j+1)"', maps to] &
P_1 \arrow[r,  "BK_{2j-1}^{\text{A}}", maps to] \arrow[r,  "(j\ \overline{j})"', maps to] &
P_2 \arrow[r,  "BK_{2j+1}^{\text{A}}", maps to] \arrow[r,  "(j+1\ \overline{j+1})"', maps to] &
P_3 \arrow[r,  "BK_{2j}^{\text{A}}", maps to] \arrow[r,  "(\overline{j}\ j+1)"', maps to] &
P_4.
\end{tikzcd}
\end{equation}
Thanks to the properties of type A Bender--Knuth involutions, the resulting pattern $P_4$ is of weight $(j~~j\!+\!1)(\shortoverline{j}~~\shortoverline{j\!+\!1}).x^{P_0}$, as desired.
However, $P_4$ does not need to be symplectic: we might find an entry $p_{lk} \ne 0$ with $2l > k$.
\begin{lem}\label{lemma}
    Let $P_4$ be defined as above. Then $p_{lk} = 0$ for all $\{(l,k)\ : \ 2l>k\}-\{(j+1, 2j)\}$.
\end{lem}
\begin{proof}
    The type A Bender--Knuth involutions from Equation \ref{eq: composition BK} only affect the rows $2j$ and $2j\pm 1$ of the pattern. 
    The value of an entry of $P_4$ in position $(l,k)$ with $2l>k$ and $k\not\in\{2j-1, 2j, 2j+1\}$ is thus 0, since $P_0$ is symplectic.
    
    Suppose $k=2j-1$. By the above, the value of $p_{lk}$ in the pattern $P_4$ is computed as
    $\min\{p_{l, k+1}, 0\} + \max\{0,0\} - 0$, which is 0. A similar argument applies to $k = 2j+1$.
    
    Let $k = 2j$, let $l > j+1$. By the above, the value of $p_{lk}$ is $\min\{0,0\}+\max\{0,0\}-0 = 0$. 
\end{proof}
That is, the only possible obstruction to the symplectic property is the value of entry $p_{j+1, 2j}$. We compose with the weight-preserving map $\rect$ (\emph{rectification}) that subtracts $p_{j+1,2j}$ from the entries $p_{j+1,2j}$, $p_{j+1,2j+1}$, $p_{j, 2j}$, and $p_{j, 2j-1}$, and is the identity everywhere else. Indeed, the fact that it is weight-preserving follows from the definition of weight of a pattern, where the variable $x_j$ is raised to the power $2S_{2j-1}-S_{2j}-S_{2j-2}$.

Define the $j$th type C Bender--Knuth involution as the composite \[
{BK^{\text{C}}_j := \rect \circ BK^{\text{A}}_{2j}
\circ BK^{\text{A}}_{2j+1}
\circ BK^{\text{A}}_{2j-1}
\circ BK^{\text{A}}_{2j}}.
\]

\begin{ej}\label{example patterns}
    Let $j=2$. We illustrate the 2nd type C Bender--Knuth involution on a symplectic pattern with 6 rows and top row $(3,3,2)$.
    \[
    \scriptsize\arraycolsep=0em\def\arraystretch{1}
    \begin{array}{CCCCCC|}
    3 & & 3 & & 2  &  \\
    & 3 & & 2 & & 0   \\
    & & \cellcolor{yellow}3 &\cellcolor{yellow} &\cellcolor{yellow} 0  & \cellcolor{yellow} \\
    & & & 2 & & 0   \\
    & & & & 1  &  \\
    & & & &  & 1  
    \end{array}
    \xrightarrow{\tiny(\bar{2}~3)}\!\!\!\!\begin{array}{CCCCCC|}
    3 & & 3 & & 2  &  \\
    & 3 & & 2 & & 0   \\
    & & \cellcolor{green!30}2 &\cellcolor{green!30} &\cellcolor{green!30} 2  & \cellcolor{green!30} \\
    & & & 2\cellcolor{yellow} & \cellcolor{yellow}& \cellcolor{yellow}0   \\
    & & & & 1  &  \\
    & & & &  & 1  
    \end{array}
    \xrightarrow{\tiny(2~\bar{2})}\!\!\!\!\begin{array}{CCCCCC|}
    3 & & 3 & & 2  &  \\
    & 3 \cellcolor{yellow}& \cellcolor{yellow}& 2 \cellcolor{yellow}& \cellcolor{yellow}& 0 \cellcolor{yellow}  \\
    & & 2 & & 2 &  \\
    & & & 2\cellcolor{green!30} & \cellcolor{green!30}& 1 \cellcolor{green!30}  \\
    & & & & 1  &  \\
    & & & &  & 1  
    \end{array}
    \xrightarrow{\tiny(3~\bar{3})}\!\!\!\!\begin{array}{CCCCCC|}
    3 & & 3 & & 2  &  \\
    & 3 \cellcolor{green!30}& \cellcolor{green!30}& 2\cellcolor{green!30} & \cellcolor{green!30}& 2\cellcolor{green!30}   \\
    & & \cellcolor{yellow}2 & \cellcolor{yellow} & \cellcolor{yellow}2 & \cellcolor{yellow}  \\
    & & & 2 & & 1   \\
    & & & & 1  &  \\
    & & & &  & 1  
    \end{array}
    \xrightarrow{\tiny(\bar{2}~3)}\!\!\!\!\begin{array}{CCCCCC|C}
    3 & & 3 & & 2  &  \\
    & 3 & & 2 & & 2   \\
    & & 3 \cellcolor{green!30}& \cellcolor{green!30}& 2 \cellcolor{green!30}&\cellcolor{green!30} & {\cellcolor{red!30}1} \\
    & & & 2 & & 1  \\
    & & & & 1  &  \\
    & & & &  & 1  
    \end{array}
    \xrightarrow{\tiny\rect}\!\!\!\!\begin{array}{CCCCCC|C}
    3 & & 3 & & 2  &  \\
    & 3 & & 2 & & \cellcolor{blue!20}1   \\
    & & 3 & & \cellcolor{blue!20}1 & & \cellcolor{blue!20}{\color{black!50}0}   \\
    & & & 2 & & \cellcolor{blue!20}0   \\
    & & & & 1  &  \\
    & & & &  & 1  
    \end{array}
    \]
    In the first step, we apply $BK^\text{A}_4$ (corresponding to $(\bar{2}~3)$). This acts on the 4th row of the pattern, sending $(3,0,0,0)$ to $(2,2,0,0)$. For instance, $3$ maps to $\min\{3\} + \max\{2,2\} - 3 = 2$. We keep doing this as indicated.

    On the fourth step, the 4th row $(2,2,0,0)$ is sent to $(3,2,1,0)$. Indeed, the first $0$ maps to $\min\{1,2\} + \max\{0,0\} - 0 = 1$. The array is no longer half-triangular. The rectification map corrects this by subtracting $1$ from the four highlighted entries of the last pattern.
\end{ej}

\begin{proof}[Proof of Proposition \ref{thm: bender-knuth C}.]
It suffices to show that each $BK^\text{C}_j$ is an involution. 
Consider \vspace{-1em}
\[\small
\begin{tikzcd}[column sep = -.8em, row sep = -0.5em]
\phantom{blahb}\Phi: &
P_0 & \xmapsto{BK_{2j}^{\text{A}}} 
& P_1 & \xmapsto{BK_{2j-1}^{\text{A}}}  
& P_2 & \xmapsto{BK_{2j+1}^{\text{A}}} 
& P_3 & \xmapsto{BK_{2j}^{\text{A}}}
& P_4 & =
& P'_5 & \xmapsto{BK_{2j}^{\text{A}}} 
& P'_6 & \xmapsto{BK_{2j-1}^{\text{A}}}  
& P'_7 & \xmapsto{BK_{2j+1}^{\text{A}}} 
& P'_8 & \xmapsto{BK_{2j}^{\text{A}}}
& P'_9 & \text{\, \, and}\\
(BK^\text{C}_j)^2: &
P_0 & \xmapsto{BK_{2j}^{\text{A}}} 
& P_1 & \xmapsto{BK_{2j-1}^{\text{A}}}  
& P_2 & \xmapsto{BK_{2j+1}^{\text{A}}} 
& P_3 & \xmapsto{BK_{2j}^{\text{A}}}
& P_4 & \xmapsto{\rect}
& P_5 & \xmapsto{BK_{2j}^{\text{A}}} 
& P_6 & \xmapsto{BK_{2j-1}^{\text{A}}}  
& P_7 & \xmapsto{BK_{2j+1}^{\text{A}}} 
& P_8 & \xmapsto{BK_{2j}^{\text{A}}}
& P_9 & \xmapsto{\rect} P_{10}.
\end{tikzcd}
\vspace{-.5em}
\]
Note that $\Phi$ is the identity, and that both maps are identical on most entries.
Indeed, in each step, the value of $BK^\text{A}_i$ on an entry only depends on the value of its four neighbours. Starting with the four entries of $P_5$ that are perturbed by $\rect$, the effect of this perturbation is only measured by the last two entries in rows $2j$ and $2j\pm 1$ of $P_{10}$.

Therefore, it is enough to show that $BK_{2}^{\text{C}}$ is an involution on a generic pattern with 6 rows.
Our strategy to tackle this final computation is to take the entry-wise differences  $P_i - P'_i$ for $i= 5, ..., 9$. We have $P_9' = P_0$, and $P_{10} - P_9'$ is seen to vanish in Section \ref{sec: appendix bender-knuth}.
\end{proof}

\begin{cor}
Orthogonal polynomials in $n$ letters are $W(B_n)$-symmetric.
\end{cor}
\begin{proof}
We have $W(B_n) = W(C_n)$. 
From the weight-preserving bijection $\SOP_n(\lambda) \to \bigcup_\mu\KSpP_n(\mu)$, we get
\(
o_\lambda =
\sum_\mu sp_\mu.
\)
The result now follows from Proposition \ref{thm: bender-knuth C}.
\end{proof}

Combinatorially, type B Bender--Knuth involutions are defined on patterns by ignoring the circles and performing type C Bender--Knuth involutions, and later placing the circles back to where they were.\\

To describe type B and C Bender--Knuth involutions on tableaux, it remains to interpret $\rect$. Consider a tableau $T_0$ and the composite
\begin{equation*}
\begin{tikzcd}[column sep = 5em]
T_0 \arrow[r,  "BK_{2j}^{\text{A}}", maps to] \arrow[r,  "(\overline{j}\ j+1)"', maps to] &
T_1 \arrow[r,  "BK_{2j-1}^{\text{A}}", maps to] \arrow[r,  "(j\ \overline{j})"', maps to] &
T_2 \arrow[r,  "BK_{2j+1}^{\text{A}}", maps to] \arrow[r,  "(j+1\ \overline{j+1})"', maps to] &
T_3 \arrow[r,  "BK_{2j}^{\text{A}}", maps to] \arrow[r,  "(\overline{j}\ j+1)"', maps to] &
T_4.
\end{tikzcd}
\end{equation*}
Lemma \ref{lemma} says $T_4$ is symplectic up to the existence of $\{j, \bar{\jmath}\}$-vertical dominoes between rows $j$ and $j+1$. (For a proof of the lemma in the language of tableaux see \cite[Prop.~5.9]{GutiTFM}.) The tableau $\rect(T_4)$ is constructed from $T_4$ by relabelling such dominoes into $\{j+1, \shortoverline{j+1}\}$-vertical dominoes, and sorting rows $j$ and $j+1$ into increasing order.

\begin{ej} Let $j=2$.
    We translate Example \ref{example patterns} to tableaux.
    \[
    \yTableau
    BK^\text{C}_2 : \,
    \ytableaushort{12{\bar{2}},33{\bar{3}},{\bar{3}}{\bar{3}}}*[*(yellow)]{2+1,2} \xmapsto{(\bar{2}\ 3)}
    \ytableaushort{123,{\bar{2}}{\bar{2}}{\bar{3}},{\bar{3}}{\bar{3}}}*[*(yellow)]{0,1}\xmapsto{(2\ \bar{2})}
    \ytableaushort{123,2{\bar{2}}{\bar{3}},{\bar{3}}{\bar{3}}}*[*(yellow)]{0,0,2}\xmapsto{(3\ \bar{3})}
    \ytableaushort{123,2{\bar{2}}{\bar{3}},33}*[*(yellow)]{2+1,0,1}\xmapsto{(\bar{2}\ 3)}
    \ytableaushort{12{\bar{2}},2{\bar{2}}{\bar{3}},{\bar{2}}3}*[*(red!30)]{0,1,1}\xmapsto{\rect}
    \ytableaushort{12{\bar{2}},{\bar{2}}3{\bar{3}},3{\bar{3}}}*[*(blue!20)]{0,1+1,1+1}
    \]
    Each of the four first maps are type A Bender--Knuth involutions, and the last map rectifies the tableau by getting rid of the highlighted $\{2,\bar{2}\}$-vertical domino.
\end{ej}
\begin{cor}
    Both symplectic polynomials and orthogonal polynomials (on $n$ letters) form an integral basis of $\Z[x_1^\pm, ..., x_n^\pm]^{S_2\wr S_n}$.
\end{cor}
\begin{proof}
    Let $\le$ be the lexicographic order on $\Par_n$. 
    Let $f = \sum_\alpha c_\alpha x^\alpha \in \Z[x_1^\pm, ..., x_n^\pm]^{S_2\wr S_n}$.
    Define the leading term of $f$ as the greatest $\lambda$ such that $c_\lambda\ne0$. Then $f - c_\lambda sp_\lambda$ has a lower leading term. Thus $\{sp_\lambda\}_{\lambda\in\Par_n}$ is spanning and linearly independent. Similarly for $\{o_\lambda\}_{\lambda\in\Par_n}$.
\end{proof}
In particular, this shows that symplectic and orthogonal polynomials are integral linear combinations of Weyl characters of representations of $\Sp_{2n}$ and $\SO_{2n+1}$, respectively.

\section{A computation}
\label{sec: appendix bender-knuth}

To complete the proof of Proposition \ref{thm: bender-knuth C}, we need to verify that the proposed map $BK_2^{\text{C}}$ is an involution on the set of symplectic patterns with 6 rows and fixed shape.

To alleviate notation, we consider a pattern $a = (a^{(6)}, ..., a^{(1)})$, and denote with $b, c, d, e, f$ the image of $a$ under the following composite maps:
\[
\begin{tikzcd}[column sep = 3em]
a \arrow[r,  "BK_{4}^{\text{A}}", maps to] \arrow[r,  "(\bar{2}~3)"', maps to] & b \arrow[r,  "BK_3^{\text{A}}", maps to] \arrow[r,  "(2~\bar{2})"', maps to] & c \arrow[r,  "BK_5^{\text{A}}", maps to] \arrow[r,  "(3~\bar{3})"', maps to] & d \arrow[r,  "BK_4^{\text{A}}", maps to] \arrow[r,  "(\bar{2}~3)"', maps to] & e \arrow[r,  "\rect", maps to] & f.
\end{tikzcd}
\]
Moreover, we set $A$ as a copy of $e$, and $A'$ as a copy of $f$, and define $B, B', ..., F'$ as follows:
\[
\arraycolsep=0.05em
\begin{array}{rccccccccccccccccccccc}
\Phi : &a & \xmapsto{(\bar{2}~3)} & b & \xmapsto{(2~\bar{2})} & c & \xmapsto{(3~\bar{3})} &  d & \xmapsto{(\bar{2}~3)} & e & =: & A & \xmapsto{(\bar{2}~3)} & B & \xmapsto{(2~\bar{2})} & C & \xmapsto{(3~\bar{3})} & D & \xmapsto{(\bar{2}~3)} & E & = & a,\\
(BK_2^{\text{C}})^2 : & a & \xmapsto{(\bar{2}~3)} & b & \xmapsto{(2~\bar{2})} & c & \xmapsto{(3~\bar{3})} & d & \xmapsto{(\bar{2}~3)} & e & \xmapsto{\rect} & A' & \xmapsto{(\bar{2}~3)} & B' & \xmapsto{(2~\bar{2})} & C' & \xmapsto{(3~\bar{3})} & D' & \xmapsto{(\bar{2}~3)} & E' & \xmapsto{\rect} & F'.
\end{array}
\]
We have $E=a$ as noted in Section \ref{sec: patterns}, and we aim to show $F' = a$.

Let us start by comparing $A$ and $A'$. We have $A'_{jk} = A_{jk}$ for all $j, k$ except for
\[
A'_{35} = A_{35}-e_{34}, \quad
A'_{24} = A_{24}-e_{34}, \quad
A'_{23} = A_{23}-e_{34}, \quad\text{and}\quad
A'_{34} = A_{34}-e_{34} = 0.
\]
We may now turn to $B$ and $B'$, in which we thus find
\[
B'_{14} = B_{14}, \quad
B'_{24} = B_{24}, \quad\text{and}\quad
B'_{34} = B_{34}.
\]
Indeed, we have
\begin{align*}
    B'_{24} &= \min\{A'_{25}, A'_{13}\} + \max\{A'_{35},A'_{23}\} - A'_{24} \\
    & = \min\{A_{25}, A_{13}\} + \max\{A_{35}-e_{34},A_{23}-e_{34}\} - (A_{24} - e_{34}) = B_{24}, \text{ and}\\
    B'_{34} &= \min\{A_{35} - e_{34}, A_{23} - e_{34}\} \\
    & = \min\{A_{35}, A_{23}\} - e_{34} \\
    & = \min\{A_{35}, A_{23}\} - A_{34} = B_{34}.
\end{align*}
In the next step, when comparing $C$ and $C'$, we therefore note
\begin{align*}
    C'_{13} & = C_{13},
    \text{ and}\\
    C'_{23} &= \min\{B_{24},B_{12}\} + B_{34} - B_{23}'\\
    &= \min\{B_{24},B_{12}\} + B_{34} - A_{23}'\\
    &= \min\{B_{24},B_{12}\} + B_{34} - (A_{23}-e_{34}) = C_{23} + e_{34}.
\end{align*}
Similarly, in $D$, $D'$,
\[
D'_{15} = D_{15},
\quad
D'_{25} = D_{25},
\quad\text{and}\quad
D'_{35} = D_{35} + e_{34}.
\]
Finally, comparing $E$ and $E'$ gives
\[
E'_{14} = E_{14},
\quad
E'_{24} = E_{24} + e_{34},
\quad\text{and}\quad
E'_{34} = E_{34} + e_{34} = a_{34} + e_{34} = e_{34}.
\]
And now, subtracting $e_{34}$ from $E'_{34}$, $E'_{24}$, $D'_{35}$ and $C'_{23}$ recovers the pattern $E$. This shows $F' = E = a$, as desired.

\begin{note}
    \label{note:summary}
Just for illustrative purposes, we give explicitly give the patterns $A', B', ..., F'$ in terms of $A, B, ..., C$ according to the computations above. To save space, we denote $x-e_{34}$ by $x^-$ and $x+e_{34}$ by $x^+$. 
\begin{gather*}
\scriptsize
\arraycolsep=0.1em\def\arraystretch{1}
\newcolumntype{D}{>{\centering\arraybackslash$} p{1.5em} <{$}}
A'=\!\!\!\!
\begin{array}{CCCCCCCCCCC}
    A_{16} & & A_{26} & & A_{36} & & 0 & & 0  & & 0\\
    & A_{15} & & A_{25} & & A_{35}^- & & 0 &  & 0 & \\
    & & A_{14} & & A_{24}^- & & A_{34}^- & & 0  & & \\
    & & & A_{13} & & A_{23}^- & & 0 &  & & \\
    & & & & A_{12} & & 0 & &  & & \\
    & & & & & A_{11} & & & & & 
\end{array}
   \quad \xmapsto{(\bar{2}\ 3)} \quad
B'=\!\!\!\!
\begin{array}{CCCCCCCCCCC}
    B_{16} & & B_{26} & & B_{36} & & 0 & & 0  & & 0\\
    & B_{15} & & B_{25} & & B_{35}^- & & 0 &  & 0 & \\
    & & B_{14} & & B_{24} & & B_{34} & & 0  & & \\
    & & & B_{13} & & B_{23}^- & & 0 &  & & \\
    & & & & B_{12} & & 0 & &  & & \\
    & & & & & B_{11} & & & & & 
\end{array}
\\
\scriptsize
\arraycolsep=0.1em\def\arraystretch{1}
\xmapsto{(2\ \bar{2})} \quad
C'=\!\!\!\!
\begin{array}{CCCCCCCCCCC}
    C_{16} & & C_{26} & & C_{36} & & 0 & & 0  & & 0\\
    & C_{15} & & C_{25} & & C_{35}^- & & 0 &  & 0 & \\
    & & C_{14} & & C_{24} & & C_{34} & & 0  & & \\
    & & & C_{13} & & C_{23}^+ & & 0 &  & & \\
    & & & & C_{12} & & 0 & &  & & \\
    & & & & & C_{11} & & & & & 
\end{array}
   \quad \xmapsto{(3\ \bar{3})} \quad
D'=\!\!\!\!
\begin{array}{CCCCCCCCCCC}
    D_{16} & & D_{26} & & D_{36} & & 0 & & 0  & & 0\\
    & D_{15} & & D_{25} & & D_{35}^+ & & 0 &  & 0 & \\
    & & D_{14} & & D_{24} & & D_{34} & & 0  & & \\
    & & & D_{13} & & D_{23}^+ & & 0 &  & & \\
    & & & & D_{12} & & 0 & &  & & \\
    & & & & & D_{11} & & & & &
\end{array}
\\
\scriptsize
\arraycolsep=0.1em\def\arraystretch{1}
\xmapsto{(3\ \bar{2})} \quad
E'=\!\!\!\!
\begin{array}{CCCCCCCCCCC}
    E_{16} & & E_{26} & & E_{36} & & 0 & & 0  & & 0\\
    & E_{15} & & E_{25} & & E_{35}^+ & & 0 &  & 0 & \\
    & & E_{14} & & E_{24}^+ & & E_{34}^+ & & 0  & & \\
    & & & E_{13} & & E_{23}^+ & & 0 &  & & \\
    & & & & E_{12} & & 0 & &  & & \\
    & & & & & E_{11} & & & & &
\end{array}
   \quad \xmapsto{\rect} \quad
F' = \!\!\!\!
\begin{array}{CCCCCCCCCCC}
    E_{16} & & E_{26} & & E_{36} & & 0 & & 0  & & 0\\
    & E_{15} & & E_{25} & & E_{35} & & 0 &  & 0 & \\
    & & E_{14} & & E_{24} & & E_{34} & & 0  & & \\
    & & & E_{13} & & E_{23} & & 0 &  & & \\
    & & & & E_{12} & & 0 & &  & & \\
    & & & & & E_{11} & & & & &
\end{array}
\end{gather*}
\end{note}


\section*{Acknowledgements}
I thank my MSc advisors Catharina Stroppel and Jacob Matherne, and my PhD advisor Mark Wildon, for helping with the construction of this text.
I am indebted to my friends Giulio Ricci, Barbara Muniz, and Chris Mili\'{o}nis for their enriching conversations.

\newpage
\section*{Addendum}

In Page 3, paragraph 4, we sketch a potential proof of Proposition \ref{thm: bender-knuth C} without the computation of $\S\ref{sec: appendix bender-knuth}$. We furthermore mention that there is a gap in the sketch. It was pointed out by Darij Grinberg \cite{GrinbergMO} that the gap is bypassed by \cite{GrinbergRoby}, giving a second proof of Proposition 4.3.\medskip

The first part is to \emph{detropicalize} the expression for Bender--Knuth involutions on GT patterns. Detropicalization refers to a map of semifields that can be described in layman terms as follows: starting with an expression in terms of ``$\max$''s and sums, change each instance of ``$\max$'' for a sum and each sum for a product. The result will be a rational function. For instance, the detropicalization of the Bender--Knuth map on GT patterns
\[
p_{i,j} \mapsto 
\min\{p_{i,j+1},p_{i-1,j-1}\} + \max\{p_{i+1,j+1},p_{i,j-1}\} - p_{i,j},
\]
is the rational map
\[
x_{i,j} \mapsto \frac{(x_{i,j+1}+x_{i-1,j-1})}{\left(x_{i+1,j+1}^{-1}+x_{i,j-1}^{-1}\right)\cdot x_{i,j}}.
\]
The detropicalization of the rectification map
\[
p \mapsto p - p_{j+1,2j}
\]
is the rational map
\[
x \mapsto \frac{x}{x_{j+1,2j}},
\]
where $p$ is one of $p_{j+1,2j}$, $p_{j+1,2j+1}$, $p_{j, 2j}$, and $p_{j, 2j-1}$, and similarly for $x$.

Combining these maps, we obtain a detropicalized type C Bender--Knuth map $\text{Trop}^{-1} BK^\text{C}_2$.
This map can be easily implemented into a symbolic algebra software; we used SageMath \cite{Sage}. It is then possible to check that $\text{Trop}^{-1} BK^\text{C}_2$ on a pattern of order 6 is an involution (as a rational function).

The second part of this argument is due to Darij Grinberg and goes as follows.
All it remains is to show that the order of $BK^\text{C}_2$ divides the order of $\text{Trop}^{-1} BK^\text{C}_2$. And this is clear as a general statement, since this is just saying that $f^2 = \mathrm{id}$ for all semifields (the $2$ here is the order of the detropicalization) implies $f^2 = \mathrm{id}$ for the tropical semifield.
See \cite[Rmk.~10, reason 2]{GrinbergRoby} for details.\medskip

Darij pointed out in private communication that the SageMath computation is not crucial: the computation of $\S\ref{sec: appendix bender-knuth}$ lifts to the detropicalization. This is clear from the picture of Note \ref{note:summary}, which lifts to a similar picture in which $x^+$ stands for $x\cdot e_{34}$ and $x^{-}$ for $x/e_{34}$ instead.

\end{document}